\documentclass[aps,twoside, a4paper, nofootinbib,10pt]{revtex4}
\usepackage{amssymb}
\usepackage[mathscr]{eucal}
\usepackage[dvips]{graphicx}
\usepackage{amsmath}
\usepackage{amsthm}
\def \ni{\noindent}

\newcommand{\be}{\begin{equation}}
\newcommand{\ee}{\end{equation}}
\newcommand{\ben}{\begin{equation*}}
\newcommand{\een}{\end{equation*}}
\newcommand{\bes}{\begin{eqnarray}}
\newcommand{\ees}{\end{eqnarray}}
\newcommand{\besn}{\begin{eqnarray*}}
\newcommand{\eesn}{\end{eqnarray*}}
\newcommand{\txt}{\textrm}
\newtheorem{theorem}{Theorem}
\newtheorem{lemma}{Lemma}
\newtheorem{definition}{Definition}
\newtheorem*{lemma*}{Lemma}

\begin{document}

\title{On the Ashtekar-Lewandowski Measure as a Restriction of the Product One}
\author{Tamer Tlas}

\begin{abstract}  \ni \textit{It is known that the $k$-dimensional Hausdorff measure on a $k$-dimensional submanifold of $\mathbb{R}^n$ is closely related to the Lebesgue measure on $\mathbb{R}^n$. We show that the Ashtekar-Lewandowski measure on the space of generalized $G$-connections for a compact, connected, semi-simple Lie group $G$, is analogously related to the product measure on the set of all $G$-valued functions on the group of loops. We also show that, the Ashtekar-Lewandowski measure is, under very mild conditions, supported on nowhere-continuous generalized connections.\\
}
\end{abstract}

\maketitle

It is a truism to state that many of the deepest mathematical ideas originated in physics. A particular instance of this fact is the infusion of quantum field theoretic methods into low dimensional topology which began with \cite{witten}. Experience has shown that quantum gauge theories are the most interesting and relevant ones for topological applications. The fundamental tool in this field is the rigorously-undefined Feynman integral over histories, which in the context of gauge theories, is an integral over the space of connections on some principal $G$-bundle modded out by gauge transformations. It might seem hopeless to define this integral. On one hand, the space of gauge equivalence classes of smooth connections is a rather complicated space; it is not even an infinite dimensional manifold. On the other hand, the difficulties in defining non-gaussian measures in the infinite dimensional setting are well-known. Fortunately, it turns out that these two difficulties solve each other in a certain sense, as there does exist an analogue of the Lebesgue measure on a suitable completion of the space of smooth connections. This measure is called the Ashtekar-Lewandowski (A-L) measure and shall be denoted by $\mu_{A-L}$ below. Let us briefly review its construction. The construction that is presented below is essentially that in \cite{baez}, but with some differences entailed by the fact that here we work with the group of loops instead of the path groupoid. \\

\section{the Ashtekar-Lewandowski Measure}

Fix a compact, connected Lie group $G$, a smooth principal $G$-bundle $P \to M$, a point $o \in M$ and a point $\ast$ in the fiber over $o$.  Let $\tilde{\mathcal{L}}$ be the set of all piecewise smooth loops which are immersive on each piece. In other words, $\tilde{\mathcal{L}}$ is the set of all maps from $I = [0,1]$ to $M$ such that if $\tilde{\gamma} \in \tilde{\mathcal{L}}$, then $\tilde{\gamma}(0) = o = \tilde{\gamma}(1)$, and there is a partition of $I$, $\{t_1, \dots, t_n \}$ such that $\tilde{\gamma}$ restricted to any interval of this partition is smooth and immersive (note that the (one-sided) derivative is required not to vanish at the endpoints of each interval). Let $\mathcal{A}_{smooth}$ stand for the set of smooth connections on $P$. Define an equivalence relation $\sim$ on $\tilde{\mathcal{L}}$ by declaring two loops equivalent if they are holonomically equivalent, i.e. if the holonomies around them are identical for any element in $\mathcal{A}_{smooth}$. Endowing $\tilde{\mathcal{L}}$ with the usual path product $\cdot$ and quotienting by $\sim$, it is easy to see that $\mathcal{L} \equiv (\tilde{\mathcal{L}} / \sim, \cdot)$ is a group. If we identify now, using $\ast$, holonomies around elements of $\mathcal{L}$ with elements of $G$, then it is straightforward to verify that any element of $\mathcal{A}_{smooth}$ is in fact a homomorphism from $\mathcal{L}$ to $G$. For reasons which will be apparent below, we shall denote the set of \textit{all} homomorphisms from $\mathcal{L}$ to $G$ by $\mathcal{A}$. \\

Let $Cyl$ stand for the algebra of cylindrical functions on $\mathcal{A}_{smooth}$, that is for the algebra of functions of the form

\ben
A \to f \big ( U(A, \gamma_1), \dots, U(A, \gamma_n)  \big ),
\een

where $A \in \mathcal{A}_{smooth}$, $U(A, \gamma_i)$ stands for the holonomy of $A$ around $\gamma_i$, which can be identified with an element of $G$, and $f$ is a continuous function from $G^n$ to $\mathbb{C}$. We shall denote the $Cyl$ element that $f$ defines by $\hat{f}$.  Let us now make two important definitions:

\begin{definition}
A finite collection of loops $\{ \gamma_1, \dots, \gamma_n \}$ is said to be holonomically independent if for any n-tuple of elements of $G$, $(g_1, \dots, g_n)$ there is $A \in \mathcal{A}_{smooth}$, such that $(g_1, \dots, g_n) = \big ( U(A, \gamma_1), \dots, U(A, \gamma_n) \big )$.
\end{definition}

\begin{definition}
A collection of loops $\{\gamma_1', \dots, \gamma_m' \}$ is said to generate the loops $\{\gamma_1, \dots, \gamma_n \}$ if the second collection is a subset of the subgroup of $\mathcal{L}$ generated by the first one.
\end{definition}

In \cite{baez} it was shown that for any finite collection of piecewise regular curves one can find another well-behaved, in a sense irrelevant for us here, collection of regular curves such that any element of the original collection is obtained from the new one by finitely many path multiplications and inverses . In \cite{lewandowski, fleischhack} it was demonstrated that the set of parallel transports of $A$ along such a well-behaved collection of $m$ curves is all of $G^m$ as $A$ ranges over $\mathcal{A}_{smooth}$ if $G$ is connected and semi-simple. It is not difficult to convert these results to the loop case. Let $K$ be the compact set which is equal to the union of the images of the paths $\gamma_1', \dots, \gamma_m'$ and let $x_1, \dots x_k$ be the set of the respective endpoints of these paths. We can find open, disjoint from $K$ and mutually disjoint sets $U_1, \dots, U_k \subset M$. For each endpoint we take a smooth, piecewise immersive path linking $o$ to the endpoint and passing through one of the $U_i$'s thus converting the paths into loops. It is obvious that the collection of loops thus obtained generates the loops from $\{\gamma_1, \dots, \gamma_n \}$. Also, this collection is still independent, since it follows from Proposition $A.1$ in \cite{fleischhack} that the connection can be adjusted on the disjoint sets $U_i$ such that its parallel transport on each of the new linking paths is trivial.\\

We therefore have that if $G$ is semi-simple, then given any finite collection of loops $\{ \gamma_1, \dots, \gamma_n \}$, there exists another finite collection $\{ \gamma_1', \dots, \gamma_m' \}$ such that the $\gamma'$'s are both holonomically independent and generate the $\gamma$'s. We shall use this fact repeatedly in what follows. Now, let $w_1, \dots, w_n$ be the words which give $\gamma_1, \dots, \gamma_n$ in terms of the alphabet $\{ \gamma_1', \dots, \gamma_m'\}$. Note that if for some $A$ the holonomies around the $\gamma'$'s are $(g_1, \dots, g_m)$, then the holonomies around the $\gamma$'s are the images of $(g_1, \dots, g_m)$ under the word maps defined by the words $w_1, \dots, w_n$. We now can define a linear operator on the set $Cyl$ via the formula:
 
 \be
 \label{eq:a-l}
 \int \hat{f} d\mu_{A-L} \equiv  \int_{G^m} f \big (  w_1 (g_1, \dots, g_m) , \dots, w_n (g_1, \dots, g_m)  \big ) \, dg_1 \dots dg_m.
 \ee

Here, $dg_i$ stands for the normalized Haar measure on the group $G$.\\

Of course a given function can be cylindrical with respect to many different families of loops. Additionally, for any such family, one has many choices for the loops $\gamma'$'s. It can be shown nonetheless that these choices are immaterial as they all give the same answer. The above definition of the operator is thus consistent. It is trivial now to extend it to the completion of $Cyl$ in the sup norm, denoted by $Cyl^{\star}$, and thus define a bounded linear transformation on this completion. Using the Gel'fand-Naimark theorem, it follows that $Cyl^{\star}$ is the set of continuous functions on a compact Hausdorff space, the spectrum of $Cyl^{\star}$. Applying now Riesz-Markov we see that this bounded linear transformation is the integral with respect to a probability measure. This measure is the A-L one. It should be clear that the domain of the A-L measure is the spectrum of $Cyl^{\star}$. This completes the review of the construction of the A-L measure.\\

The purpose of this paper is to give an alternative construction of this important measure. The starting point is to note that the spectrum of $Cyl^{\star}$ can be identified with $\mathcal{A}$, if the latter space is given an appropriate topology. To see this, start by noting that $\mathcal{A} \subset \mathcal{F}$ where $\mathcal{F}$ is the space of all functions from $\mathcal{L}$ to $G$. In turn, $\mathcal{F}$ is equal to the product space $\prod_{\gamma \in \mathcal{L}} G_{\gamma}$. Topologize the product space with the product topology. Below, we shall only discuss the case when $G$ is a compact and semi-simple. It follows by Tychonoff's theorem that in this case $\mathcal{F}$ is a compact space. Moreover, since $G$ is a Lie group and thus is Hausdorff, we have that $\mathcal{F}$ is also Hausdorff. The topology we shall put on $\mathcal{A}$ is the subspace one induced from $\mathcal{F}$. Note that $\mathcal{A}$ is in fact a closed set. To see this, take any net of elements $\phi_{\alpha}$ in $\mathcal{A}$ converging to some element $\phi \in \mathcal{F}$, and let $\gamma_1$ and $\gamma_2$ be two loops. From the definition of the product topology, the fact that each $\phi_{\alpha}$ is a homomorphism and the fact that multiplication is continuous in $G$, we have that$$\phi(\gamma_1 \gamma_2) = \lim \phi_{\alpha}(\gamma_1 \gamma_2) = \lim \phi_{\alpha}(\gamma_1) \lim \phi_{\alpha} (\gamma_2) = \phi(\gamma_1) \phi(\gamma_2).$$ Therefore, since a closed subspace of a compact space is compact, we have that $\mathcal{A}$ is also a compact space. We can now state 

\begin{lemma}
$Cyl^{\ast}$ is isometrically isomorphic to the algebra of continuous functions on $\mathcal{A}$.
\end{lemma}

\begin{proof}
We shall show that given any $\alpha \in \mathcal{A}$ and finitely many loops $\gamma_1, \dots, \gamma_n$ then there is a $\beta \in \mathcal{A}_{smooth}$ such that $\beta(\gamma_i) = \alpha(\gamma_i)$, (incidentally, this shows that $\mathcal{A}_{smooth}$ is dense in $\mathcal{A}$ but we shall not need this fact here). As mentioned above, there are loops $\gamma_1', \dots, \gamma_m'$ such that every one of $\gamma_1, \dots, \gamma_n$ is a product of the $\gamma'$'s and that for any $(g_1, \dots, g_m) \in G^m$ there is a smooth connection $A$ such that $(g_1, \dots, g_m) = (U(A, \gamma_1'), \dots, U(A, \gamma_m'))$. Therefore there is an element $\beta \in \mathcal{A}_{smooth}$ such that $\beta(\gamma_j') = \alpha(\gamma_j')$ for $j =1, \dots, m$. It follows from the fact that both $\alpha$ and $\beta$ are homomorphisms that $\beta(\gamma_i) = \alpha (\gamma_i)$ for $i=1, \dots, n$.\\

Let $F$ stand for the algebra of continuous, complex-valued functions on $\mathcal{A}$, and let $F_0$ be the subalgebra consisting of those functions which depend on only finitely many components of $\mathcal{A}$. It is trivial to see that $F_0$ is closed under conjugation, vanishes nowhere and separates points, and thus, by Stone-Weierstrass theorem, $F_0$ is dense in $F$. Note that any element $f \in F_0$ can be written as $\bar{f}\circ \pi_{\gamma_1, \dots, \gamma_n}$ where $\bar{f}$ is a continuous function from $G^n$ to $\mathbb{C}$ and $\pi_{\gamma_1, \dots, \gamma_n}$ is (the restriction to $\mathcal{A}$ of) the projection to the components corresponding to the loops $\gamma_1, \dots, \gamma_n$. \\

Let the map $\iota : F_0 \to Cyl$ be defined in the obvious way:

\ben
\iota (f) = \bar{f} \big ( U(. , \gamma_1) , \dots, U(. , \gamma_n) \big).
\een

It is easy to see that $\iota$ is a $*$-homomorphism. To see that $\iota$ is 1-1, assume $f \neq 0$ and suppose that $f$ depends only on the components corresponding to $\gamma_1, \dots, \gamma_n$. We thus have that $f(\alpha) \neq 0$ for some $\alpha \in \mathcal{A}$. It follows from the above discussion that there is a $\beta \in \mathcal{A}_{smooth}$ such that $\beta(\gamma_i) = \alpha(\gamma_i)$ for $i =1, \dots, n$ and thus that $f(\alpha) = f(\beta)$. But this means that $\iota (f )(\beta) \neq 0$ and thus $\iota$ is an injection. For surjectivity, let $\hat{f}$ be an element in $Cyl$ depending only on the holonomies along $\gamma_1, \dots, \gamma_n$ through a function $f$. Then note that $ \iota (f \circ \pi _{\gamma_1, \dots, \gamma_n}) = \hat{f}$ and $\iota$ is onto.\\

Now, the same proof as above when we showed that $\mathcal{A}_{smooth}$ is dense in $\mathcal{A}$ shows that the following equality between sets holds:

\ben
\{ ( (\alpha (\gamma_1), \dots, \alpha(\gamma_n) )  : \alpha \in \mathcal{A} \} = \{ (U(A, \gamma_1) , \dots, U(A, \gamma_n) ) : A \in \mathcal{A}_{smooth}  \}.
\een

It follows from this equality that $\iota$ is an isometry, and thus it extends to an isometry from $F$ to $Cyl^{\ast}$.
\end{proof}

Note that the above lemma implies at once that the spectrum of $Cyl^{\ast}$ can be naturally identified with $\mathcal{A}$. In view of this identification, the A-L measure can be considered to be a measure defined on $\mathcal{A}$.\\

Now, as was mentioned previously, $\mathcal{A} \subset \mathcal{F}$. However, thinking of $\mathcal{F}$ as a product space, it is obvious that $\mathcal{F}$ can be endowed with the product measure, where the measure on the individual factors of $G$ forming $\mathcal{F}$ is just the Haar one. Let us denote this measure by $\mu_{prod}$. Our goal is to relate $\mu_{A-L}$ to $\mu_{prod}$. \\

Before we do this let us clarify a technical point regarding $\mu_{prod}$ on $\mathcal{F}$. Usually, this measure would be defined on the product $\sigma-$algebra, which is the sigma algebra generated by sets of the form $\prod_{\gamma \in \mathcal{L}} B_\gamma$ where each $B_\gamma$ belongs to the Borel $\sigma-$algebra of $G$ and all but finitely many of these factors are equal to $G$ itself. This product $\sigma-$algebra is strictly smaller than the Borel sigma algebra on $\mathcal{F}$ coming from the product topology. However, it can be shown that the Lebesgue completion of the product algebra in our case contains the Borel $\sigma-$algebra on $\mathcal{F}$ (see, e.g., Theorem 7.14.3 in \cite{bogachev}, volume II). Thus $\mu_{prod}$ admits a unique extension to the Borel $\sigma-$algebra. It is this unique extension which is denoted by $\mu_{prod}$ in this paper. The reader who is unfamiliar with product measures should consult the relevant sections of \cite{bogachev}, in particular those dealing with products of measures and Fubini's theorem (Theorem 3.4.4 and Lemma 3.5.2 in \cite{bogachev}, volume I).

\section{The General Construction}

\subsection{Motivation}

What could the relationship between $\mu_{A-L}$ and $\mu_{prod}$ be? The most obvious thing to try is to check whether $\mu_{A-L}$ is simply the restriction of $\mu_{prod}$ to $\mathcal{A}$. However, since all the elements of $\mathcal{A}$ take the same value at the trivial loop (this value being the identity in $G$), it follows at once that $\mu_{prod} (\mathcal{A}) = 0$, and the restriction of $\mu_{prod}$ is a trivial measure and not a probability one.\\

We are thus faced with the problem of restricting a measure to a compact null subset. To gain some intuition, let us consider the problem of restricting the Lebesgue measure in $\mathbb{R}^n$ to a null subset. In this situation, we have in fact several, not unrelated, ways to proceed. Let us consider them and see if we can adapt one of them to the situation at hand.\\

One obvious possibility is to note that $\mathbb{R}^n$ carries a Riemannian metric and the measure on it is induced from this metric. Thus, if our null subset is a smooth submanifold, we could `restrict' the Lebesgue measure to it by pulling back the metric and then using it to construct the measure. This procedure seems to be inapplicable in our case since $\mathcal{F}$ is not a metrizable space, and thus on one hand, it is not clear how to make a smooth manifold out of it, and on the other hand, it does not carry a Riemannian metric. Additionally, it is not at all obvious that $\mathcal{A}$ is going to be a smooth submanifold. Finally, even if these problems are somehow circumvented, it is not obvious how to generalize to infinite dimensions the procedure which gives a measure from a Riemannian metric.\\

Now, while restricting the metric seems to hit an impasse, the discussion above does suggest another way to proceed. The point is to note that the above recipe gives, for a submanifold, nothing but (a multiple of) the Hausdorff measure of the appropriate dimension (the reader who is unfamiliar with Hausdorff measures should consult e.g. the last chapter of \cite{folland} for an introduction and \cite{federer} for a more sophisticated treatment). Perhaps we should try to generalize the Hausdorff measure construction to our case? The advantage, of course, being that we do not need any smooth structure on our sets in order to perform this construction. Unfortunately, this way of proceeding has its own shortcomings. First, one still needs a metric on the space under consideration (it does not have to be a Riemannian one though), or at least a way to assign volume to `balls' (defining what a ball is in a non-metrizable space is another issue). Also, one needs to know the Hausdorff dimension of the subset in advance in order pick the right Hausdorff measure, and the fact that the Hausdorff dimension of $\mathcal{A}$ is expected to be infinite is another source of difficulty. Finally, even if one finds a way to bypass the problems above, the resulting measure on the subset may still be a trivial one as there are sets whose Hausdorff dimension is $\alpha$ but whose $\alpha$-th Hausdorff measure is zero.\\

Nonetheless, the idea of using the Hausdorff measure is certainly appealing intuitively. We shall see that it is possible to adapt the construction of the Hausdorff measure in a way that will avoid essentially all the problems mentioned in the previous paragraph. The key fact that we need is that for `nice' sets, the Hausdorff measure can be induced in the following way \cite{federer}:\\

If $A \subset C \subset \mathbb{R}^n$, where $C$ is compact and $\alpha$-rectifiable, and $A$ is a Borel subset of $C$, then the $\alpha$-th Hausdorff measure of $A$ is equal to 

\be \label{eq:first} \lim_{\delta \to 0} \frac{|A_\delta|}{\delta^{n-\alpha}},\ee

where $|.|$ is the Lebesgue measure in $\mathbb{R}^n$ and $A_\delta$ is the set of all points in $\mathbb{R}^n$ whose distance to $A$ is less than $\delta$. Note that in general the expression above will only be proportional to the Hausdorff measure. However, the constant of proportionality will be the same for all subsets $A$. This constant depends on the normalization chosen for the Hausdorff measure. The reader may assume that we have made our choice so as to make this constant equal to 1.\\

Now, this may not seem to be much of an improvement over the usual procedure. First, one still needs to know the appropriate Hausdorff dimension ($\alpha$ above) in advance. Second, this dimension is most likely to be infinite in our case.  Third, one uses the metric in order to construct $A_\delta$. Thus, the above formula is still not good enough and we need to further modify or reformulate it. \\

To handle the first two problems, let us replace the denominator with $|C_\delta|$, i.e. by the set of all points in $\mathbb{R}^n$ within distance $\delta$ of $C$. The motivation for this is that for sufficiently `tame' $C$ (e.g. if $C$ is a submanifold) $|C_\delta|$ does not differ by much from the Hausdorff measure of $C$ multiplied by $\delta^{n- \alpha}$ for sufficiently small $\delta$'s. Additionally, this will normalize the resulting measure so it becomes a probability one. \\

It is not difficult now to handle the third issue. All we need to do is to note that $A_\delta$ is just an open neighborhood of $A$, and thus we can simply try to modify the above formula by replacing $A_\delta$ with some other class of open neighborhoods, one that does not require a metric to single it out. For example, the collection of all open neighborhoods of $A$ is a possible choice. From now on, we shall call any neighborhood of some set a `thickening' of the set.\\

Thus we arrive at a tentative definition of the restriction that we are looking for: \\

\textit{Let the restricted measure of $A$ be the limit of the ratio of the original measure of its thickening to the original measure of the thickening of $C$.}\\

This definition, while clearly applicable in our case, seems to have created a whole host of new issues to deal with, as it is not clear which thickenings to use, why does the limit exist, and finally why does it define a measure.\\

\subsection{Restricting Measures}

It turns out to be more convenient to avoid trying to make sense of the definition above, and to instead construct the integral first, and only obtain the measure later, via Riesz-Markov theorem or via Daniell construction. This route also offers the advantage of making the comparison with the A-L measure much simpler, as it would be sufficient to show that the integrals of cylindrical functions are the same with respect to both measures. \\

Therefore, let $f: C \to \mathbb{R}$ be a continuous function. Let $\{ C_i \}_{i \in \mathcal{I}}$ be a collection of thickenings of $C$ indexed by some set $\mathcal{I}$. We will want to take some kind of a limit in the end as $C_i$'s `shrink' to $C$. Therefore, let us assume that $I$ is in fact a directed set, and thus $C_i$ is a net of thickenings. Since $\mathbb{R}^n$ is normal and $C$ is closed, $f$ by Tietze's theorem admits an extension to all of $\mathbb{R}^n$, $\tilde{f} : \mathbb{R}^n \to \mathbb{R}$. We can now attempt to define the $\int_C f$ to be the limit of 

\be
\label{eq:second}
\frac{1}{|C_i|}\int_{C_i} \tilde{f}.
\ee

Now, while the above expression seems to generalize naturally to our case ($\mathcal{F}$ being compact and Hausdorff is normal and thus $\tilde{f}$ can be defined), there are two problems with it that must be solved:

\begin{itemize}
\item Are we guaranteed that the limit will exist? This might seem to be a big problem especially in view of the fact that (\ref{eq:first}), which was used as a motivation for (\ref{eq:second}), is only valid for rather special sets (e.g. rectifiable). Additionally, the problem of the existence of the limit is tied to the choice of the net of thickenings. It is not hard to see that it is possible to choose a sequence of thickenings of a two point set such that (\ref{eq:second}) will not have a limit.\\

The easiest way to solve these issues is to realize that we do \textit{not} need to solve them! Indeed, we only need to obtain from the net given by (\ref{eq:second}) a number, such that the result would satisfy linearity, positivity and which would be equal to the limit of the net when it does converge. For then, we will be able to apply Riesz-Markov or Daniell. In other words, we just need a linear positive functional on the space of bounded nets which extends the operation of taking a limit of a convergent net. Such objects are known to exist, we will pick one and will denote it by $L$ (the quickest way to show existence of such an object is via the Hahn-Banach theorem \cite{reed}). More precisely, let $\mathcal{N}$ be the $\mathbb{R}$-vector space of real-valued, bounded nets and $\mathcal{N}_c$ the subspace of convergent ones. Then $L$ is a linear, positive functional on $\mathcal{N}$ which satisfies $L(n_i) = \lim n_i$ when $n_i \in \mathcal{N}_c$.

 In view of the above we have

\be
\label{eq:third}
\int_C f = L \Big  (  \frac{1}{|C_i|}\int_{C_i} \tilde{f}
   \Big ).
\ee
\item The second issue we need to handle is to show that the right hand side of (\ref{eq:third}) is independent of the choice of the extension of $f$ used. The reason we need independence is that otherwise linearity will in general fail.\\

At this moment we shall switch the discussion from restricting the Lebesgue measure to null, compact subsets of $\mathbb{R}^n$ to the case of restricting a probability measure on a normal, compact space $X$ to a closed, null subset $C$. We shall still use ($\ref{eq:third}$) as a tentative definition for the integral, via a net of thickenings (open neighbourhoods of $C$) $\{C_i\}_{i \in \mathcal{I}}$ and via an extension $\tilde{f}$ of $f$.\\

Let us now make 

\begin{definition}
We call the net $\{ C_i \}_{i \in \mathcal{I}}$ \textit{nicely shrinking to $C$} if given any open set $U$ containing $C$, there exist an $i \in \mathcal{I}$ such that $\forall j \geq i$ we have $C_j \subset U$.
\end{definition}

It is not difficult now to prove the following

\begin{lemma}
If the net $\{ C_i \}_{i \in \mathcal{I}}$ is nicely shrinking, then the right-hand side of (\ref{eq:third}) is independent of the extension chosen for $f$ and defines a Borel, probability measure on $C$.
\end{lemma}

\begin{proof}
Assume $\{ C_i \}_{i \in \mathcal{I}}$ is nicely shrinking and that there are two extensions $\tilde{f}_1$ and $\tilde{f}_2$, and let $\epsilon >0$. Let $U \supset C$ be the preimage of $(-\epsilon, \epsilon)$ under $\tilde{f}_1 - \tilde{f}_2$. By assumption, there is a $C_i \subset U$ and thus

\ben
\bigg | \frac{1}{|C_i|} \int_{C_i} ( \tilde{f}_1 - \tilde{f}_2)  \bigg | < \epsilon
\een

This equation will be true for all the $C_j$'s for $j \geq i$. It follows that $\Big| L \Big (  \frac{1}{|C_i|} \int_{C_i} ( \tilde{f}_1 - \tilde{f}_2) \Big ) \Big | < \epsilon$, and thus $\int_C f$ is independent of the extension. As a consequence, it follows at once that $\int_C f$ is linear. This combined with the fact that $|\tilde{f}|$ can be assumed to be bounded from above by $|| f ||_\infty$ implies by Riesz-Markov that the right hand side of (\ref{eq:third}) does indeed define an integral of $f$ with respect to a Borel measure on $C$. Finally, due to the normalization chosen, the measure is a probability one. This completes the proof. 
\end{proof}
\end{itemize}

The reader should note that $\mathcal{F}, \mathcal{A}$ and $\mu_{prod}$ do satisfy the conditions needed in the above lemma, and thus given a nicely shrinking net  of thickenings $\{ \mathcal{A}_i \}_{i \in \mathcal{I}}$, we can define

\be
\label{eq:restrict}
\int_\mathcal{A} f = L \Bigg ( \frac{1}{\mu_{prod} (\mathcal{A}_i )} \, \, \int_{\mathcal{A}_i} \tilde{f}  \, \, d \mu_{prod} \Bigg ).
\ee

The above formula solves the problem of restricting the product measure on $\mathcal{F}$ to $\mathcal{A}$, provided a suitable shrinking net of thickenings of $\mathcal{A}$ can be found. Moreover, we want to choose the net in such a way so as the left hand side of (\ref{eq:restrict}) is the integral with respect to the A-L measure. Is there a net which satisfies these two conditions? The answer is yes and shall be given below. Let us remark that it will turn out that the result will turn out to be independent of the choice of $L$ due to the fact that in our case the relevant nets (coming from cylindrical functions) shall be convergent.

\section{The Ashtekar-Lewandowski measure as a restriction of the product one}

Let us now demonstrate that there is a net that will give the Ashtekar-Lewandowski measure through (\ref{eq:restrict}). This is a crucial part of the construction as different choices of the net may drastically change the resulting measure. As an illustration, let us go back to (\ref{eq:third}), and let $C$ be a smooth submanifold, and $C_i = \bigcup_{x \in C} \frac{1}{i}O$, $i = 1, 2, \dots$, where $O$ is some bounded fixed open set ($\frac{1}{i} O$ is just a scaled version of $O$). Then the measure thus obtained in the case when $O$ is a ball \textit{differs} in general from that obtained when $O$ is a cube! Thus, one must be quite careful with the choice of the net if one aims to obtain a certain specific measure on the subset. Needless to say, if the goal is to obtain \textit{some} measure on $\mathcal{A}$ then probably the most natural choice for the net would be to let $\mathcal{I}$ be the collection of basis neighbourhoods for the identity element in $\mathcal{F}$ (note that $\mathcal{F}$ is naturally a topological group under pointwise multiplication) ordered by reverse inclusion (the smaller set is the `later' index). Then, if, for $i \in \mathcal{I}$, we let $\mathcal{A}_i = \bigcup_{x \in \mathcal{A}} x . i$, it is not difficult to see, by an argument similar to that in Theorem \ref{theorem:restrict}, that $\{ \mathcal{A}_i \}_{i \in \mathcal{I} }$ nicely shrinks to $\mathcal{A}$ and thus (\ref{eq:restrict}) gives a measure. However, the measure thus obtained does not seem to be the A-L one, hence the construction given in the text. \\

In order to describe the net that we want, we need to introduce the following

\begin{definition}
Let $X$ and $Y$ be Hausdorff, compact, measure spaces, with Borel, probability measures $\mu_x$ and $\mu_y$, and suppose that $X$ is second countable. Suppose $A \subset X \times Y$ is measurable (with respect to the product Borel measure), and define $A_x = \{ y \in Y : (x,y) \in A \}$. We will call $A_x$ a `slice' of $A$ at $x$. We shall say that $A$ is `equisliceable' as a function of $x$ if all its slices have the same measure, i.e. if $\mu_y (A_x)$ is a constant independent of $x$.
\end{definition}

It should be noted that the conditions that have been put  above on the spaces ensure that there is no mismatch between the product Borel $\sigma$-algebra and the Borel $\sigma$-algebra on the product. See 6.4 in volume II of \cite{bogachev}. Also, note that we do not need to require separately the measurability of $A_x$, this is because a slice of a Borel set is Borel (the quickest way to see this is to consider the collection of all sets whose slice is Borel, and to note that this collection is a $\sigma$-algebra containing all open sets, and thus contains the Borel $\sigma$-algebra).\\

Let now the set of all finite collections of elements of $\mathcal{L}$ be denoted by $\mathcal{J}$. Let $\tilde{\mathcal{I}}$ be the subset of  $\mathcal{J} \times \mathcal{J}$ consisting of those pairs of collections such that the second one generates the first while being at the same time holonomically independent. Order $\tilde{\mathcal{I}}$ by inclusion on the first pair (i.e. $(\xi_1, \eta_1) \prec (\xi_2, \eta_2) \iff \xi_1 \subset \xi_2$). Let $\mathcal{I} = \tilde{\mathcal{I}} \times (0,\infty)$ be ordered such that  $( i_1, r_1) \prec ( i_2, r_2) \iff i_1 \prec i_2 \, \, \txt{and} \, \, r_1 > r_2$. Note that $\mathcal{I}$ is a directed set. \\

We can now state

\begin{theorem}
\label{theorem:restrict}
There exists a nicely-shrinking net of thickenings of $\mathcal{A}$ indexed by $\mathcal{I}$, $\{ \mathcal{A}_i \}_{i \in \mathcal{I}}$, such that each thickening $\mathcal{A}_{i = (\xi, \eta), r}$ is equisliceable as a function of the values of the components corresponding to $\eta$. The measure that this net defines via (\ref{eq:restrict}) is the Ashtekar-Lewandowski one.
\end{theorem}

\begin{proof}
Fix the collections $\xi = \{ \gamma_1, \dots, \gamma_n \}$ and $\eta = \{ \gamma_1', \dots, \gamma'_m \}$ where $\eta$ generates $\xi$ and $\eta$ is holonomically independent. Let $w_1, \dots, w_n$ be the words which give the $\gamma$'s in terms of $\gamma'$'s. For any $m$-tuple $g = (g_1, \dots, g_m) \in G^m$, let $w(g) \in G^n$ stand for the image of $g$ under $(w_1, \dots, w_n)$. Since $G$ is semi-simple, it has a natural invariant Riemannian metric, denoted by $\rho$. Let $\rho_k$ be the corresponding Riemannian metric on $G^k$. We shall also denote by $\rho_k$ the metric induced on $G^k$ from the Riemannian metric. Let $B ( x , r)$ stand for the ball in $G^k$ of center $x$ and radius $r$. Finally, we shall denote the Haar measure of a set $B \subset G^k$ by $|B|$. Note that since $\rho_k$ is translation-invariant, we have that $|B(x,r)|$ is independent of $x$.   \\

Define now for any $r >0$ the set $\mathcal{A}_{(\xi, \eta), r  }$ to be the subset of $\mathcal{F}$ given by 

\ben
\mathcal{A}_{(\xi, \eta), r}  \equiv    \bigsqcup_{ g \in G^m }  \bigg (  \overbrace{  \{ g \} }^{\eta}   \times   \overbrace{ B \big ( w(g) , r \big ) }^{\xi}  \times    \prod_{\gamma \notin ( \xi \cup \eta )} G_{\gamma} \bigg ).
\een

The first factor in brackets corresponds to the components of $\mathcal{F}$ corresponding to the elements of $\eta$. The second one to those corresponding to the elements of $\xi$. The reader is urged at this point to take a simple case for $\xi$ and $\eta$ (e.g. $\eta = \{ \gamma_1 \}$ and $\xi = \{ \gamma_1^2 \}$) and go through the subsequent arguments with the simple case in mind.\\

Let us prove that $\mathcal{A}_{(\xi, \eta), r}$ is open. It is clear that it is sufficient to prove that $$  \bigsqcup_{ g \in G^m }  \big (   \{ g \}    \times   B \big ( w(g) , r \big )  \big ) $$ is an open subset of $G^{m+n}$. Thus let $(x,y)  \in G^{m+n}$ be in this set. Let $$\delta = r - \rho_n\big( w(x), y) \big ).$$ Using uniform continuity of the word maps, find some $\delta' < \frac{\delta}{4}$ such that for all $g,h \in G^n$ $$\rho_m(g,h) < \delta' \implies \rho_n \big( w(g), w(h) \big) < \frac{\delta}{4}.$$ It follows that $B \big( (x,y) , \delta' \big) $ is contained in $\mathcal{A}_{(\xi, \eta), r}$ for if $(u,v) \in B\big((x,y), \delta' \big)$ then
\besn
\rho_n (v , w(u)) & \leq & \rho_n (v, y) + \rho_n (y, w(x)) + \rho_n (w(x) , w(u)) \\
& < & \frac{\delta}{4}+ (r - \delta) + \frac{\delta}{4}\\
&  = & r - \frac{\delta}{2}  <  r.
\eesn

Thus $\mathcal{A}_{(\xi, \eta), r}$ is open. We have used above the following fact from Riemannian geometry 
\be
\label{eq:inequality}
\rho_{m+n} \Big ( (x,y) , (u,v) \Big ) < \delta'   \implies \rho_m(x,u) < \delta' \, \, , \, \, \rho_n(y,v) < \delta' . 
\ee   

Here is a sketch of the proof of this fact: Let $M_1$ and $M_2$ be two Riemannian manifolds with metrics $g_1$ and $g_2$. Consider two points $(x_1, x_2)$ and $(y_1,y_2)$ in $M_1 \times M_2$, and let $\gamma_1$ and $\gamma_2$ be the two minimizing geodesics linking $x_1$ with $y_1$ and $x_2$ with $y_2$ respectively. Then by considering the energy functional it is easy to see that $\gamma = (\gamma_1, \gamma_2)$ is a minimizing geodesic in $M_1 \times M_2$ linking the two points with respect to the metric $g_1 \oplus g_2$. We then have $$\txt{length} (\gamma)  =  \int \sqrt{ (g_1 \oplus g_2) (\dot{\gamma} , \dot{\gamma})    }  
 =  \int \sqrt{ g_1(\dot{\gamma_1}, \dot{\gamma_1}) + g_2 (\dot{\gamma_2}, \dot{\gamma_2})    }
 \geq  \int \sqrt{ g_1(\dot{\gamma_1}, \dot{\gamma_1}) }
 =  \txt{length}(\gamma_1).$$ Implication (\ref{eq:inequality}) now follows since the metric on a compact Riemannian manifold coincides with the length of a minimizing geodesic.
\\

Since $\mathcal{A}_{(\xi, \eta), r}$ obviously includes the set $$\bigsqcup_{ g \in G^m }  \bigg (  \overbrace{ \{ g \} }^\eta \times \overbrace{\{ w(g) \}}^\xi   \times    \prod_{\gamma \notin ( \xi \cup \eta )} G_{\gamma} \bigg ),$$ it contains $\mathcal{A}$ and is thus a thickening of $\mathcal{A}$. Finally, the fact that $\mathcal{A}_{(\xi, \eta), r}$ is equisliceable as a function of $g$ is immediate, as the measure of each slice is equal to $|B \big ( w(g), r\big )|$ and is thus independent of $g$ (the set $X$ from Definition 4 is $G^m$ here, while the set $Y$ is $G^n \times \prod_{\gamma \notin \xi \cup \eta} G_{\gamma} $)\\

Let us show that the net defined above nicely-shrinks to $\mathcal{A}$. Note that 

\ben
\mathcal{A}_{(\xi, \eta), r} \subset \bigcup_{x \in \mathcal{A} } x \cdot U_{ (\xi, \eta),r},
\een

where $$U_{(\xi,\eta),r} = B(e,r) \times B ( e, r) \times \prod_{\gamma \notin  (\xi \cup \eta) } G_{\gamma}.$$ The first ball above is centered around the identity in $G^m$ (corresponding to the the elements of $\eta$) while the second one is centered around the identity in $G^n$ (corresponding to the elements of $\xi$). The letter $e$ stands for the identity in $G^k$ where $k$ should be clear from the context. We have also used the fact that $\mathcal{F}$ is a topological group under pointwise multiplication to define the translated set $x \cdot U_{(\xi,\eta),r}$. \\

Note that the collection of all $U_{(\xi, \eta),r}$'s forms a basis of neighbourhoods around the identity element in $\mathcal{F}$. We now utilize 

\begin{lemma}
If $\mathcal{G}$ is some compact, topological group, then given any open cover $\{ O_j \}_{j \in J}$ of $\mathcal{G}$ we can find a basis neighborhood of the identity, $V$, such that $x\cdot V \subset O_{j_x}$ for each $x \in \mathcal{G}$.
\end{lemma}

This statement is analogous to saying that any open cover of a compact set in a metric space has a Lebesgue number. We defer the proof of this lemma till later. Now, applying its statement to the open cover $\{ O, \mathcal{F} - \mathcal{A} \}$, where $O$ is an open neighbourhood of $\mathcal{A}$, we would get that there is a neighborhood of the identity $V$, such that for every $x \in \mathcal{A}$, $x \cdot V \subset O$ (since it cannot be contained in $\mathcal{F} - \mathcal{A}$). Then, by taking $\xi$ to be large enough and $r$ small enough, we would have that $$\mathcal{A}_{(\xi, \eta), r} \subset \bigcup_{x \in \mathcal{A}} x \cdot U_{(\xi, \eta),r} \subset \bigcup_{x \in \mathcal{A}} x \cdot V \subset O,$$ and thus the net would nicely-shrink to $\mathcal{A}$. \\

It remains to prove that the measure associated to this net is the A-L one. It is enough to show that the integrals of cylindrical functions are given by (\ref{eq:a-l}). Let $\hat{f} \in Cyl$ be of the form $ A \to \bar{f} \big ( U(A, \gamma_1), \dots, U (A, \gamma_n) \big )$ where $\bar{f}$ is a continuous function on $G^n$. Denote by $f$ the element of $F_0$ corresponding to $\hat{f}$ via the map $\iota$ from Lemma 1. Let $\xi = \{ \gamma_1, \dots, \gamma_n \}$ and let $\eta = \{ \gamma_1', \dots, \gamma_m' \}$ be a collection of holonomically independent loops which generate the elements of $\xi$. Fix $\epsilon > 0$. Choose the extension $\tilde{f}$ of $f$ to be the one given by $\tilde{f}(\alpha) = \bar{f} ( w ( \pi_\eta (\alpha)) )$ where $\pi_\eta : \mathcal{F} \to G^m$ is the projection on the components corresponding to $\eta$, $w$ is the same as before and $\alpha \in \mathcal{F}$. Choose $r_0 >0$ to be such that $$\forall h, k \in G^m \, \, : \, \rho_m ( h , k) < r_0 \implies | \bar{f}(w(h)) - \bar{f}(w(k)) | < \epsilon.$$ Let $ j \in \mathcal{I}$ be such that $$j = ((\xi'', \eta''), r ) \succ \Big( \big((\xi \cup \eta), \eta' \big) , r_0  \Big).$$ Letting $\tilde{f_0}$ be the function which is equal to $\tilde{f}$ on $\mathcal{A}_j$ and zero elsewhere, we have by Fubini (c.f. Lemma 3.5.2 in \cite{bogachev} volume I):

\besn
\frac{1}{\mu_{prod} (\mathcal{A}_j)} \int_{\mathcal{A}_j} \tilde{f} \: d \mu_{prod} & = & 
\frac{1}{\mu_{prod} (\mathcal{A}_j)} \int_{\mathcal{F}} \tilde{f_0} \: d \mu_{prod} \\
& = & 
\frac{1}{\mu_{prod}(\mathcal{A}_j)} \int_{G_{\eta''} } dg_{\eta''}\int_{ G_{\eta''^c}} dg_{\eta''^c} \tilde{f_0} \\
& = & \frac{1}{\mu_{prod}(\mathcal{A}_j)} \int_{G_{\eta''} } dg_{\eta''}\int_{ \pi_{\eta''^c} \big ( \mathcal{A}_j \cap P_{g_{\eta''}}  \big ) } dg_{\eta''^c} \tilde{f},\\
\eesn

where $G_{\eta''}$ is the product of copies of $G$ one for each element of $\eta''$, $dg_{\eta''}$ is the Haar measure on this product, $dg_{\eta''^c}$ is the product measure on the product of copies of $G$ one for each element in the complement $\eta''^c$ of $\eta''$ in $\mathcal{L}$. This product is naturally denoted by $G_{\eta''^c}$. Also, $P_{g_{\eta''}}$ is the subset of $\mathcal{F}$ whose $\eta''$ components have the given fixed value $g_{\eta''}$. Finally, $\pi_{\eta''^c}$ is the projection of $\mathcal{F}$ onto $G_{\eta''^c}$. Note that $ \pi_{\eta''^c} \big ( \mathcal{A}_j \cap P_{g_{\eta''}}  \big )$ is a viable domain of integration as it is a slice of Borel set ($\mathcal{A}_j$) and is thus Borel (alternatively, in this particular case it is easy to see that it is open).
\\

Let us denote by $w_1', \dots, w_m'$ the words which give $\eta$ in terms of $\eta''$, and $w'$ the word map they define from $G_{\eta''}$ to $G_{\eta}$, where $G_\eta$ is of course the product of copies of $G$ one for each element of $\eta$. It follows from (\ref{eq:inequality}) and by construction of $\mathcal{A}_j$ (note that $\xi \subset \xi''$) that $$h \in \pi_{\eta} \big ( \mathcal{A}_j \cap P_{g_{\eta''}}) \implies \rho_m ( h , w'(g_{\eta''}   ) )< r_0,$$ where $\pi_{\eta}$ stands for the projection of $\mathcal{F}$ onto the components corresponding to $\eta$. We thus have that for all such $h$ that $$\Big |\bar{f} (w(h)) - \bar{f} \big((w \circ w'\big) ( g_{\eta''}    )  )   \Big | < \epsilon,$$ where $\pi_{\eta''}$ stands for the projection of $\mathcal{F}$ onto the components corresponding to $\eta''$. But this means that, since $\mathcal{A}_j$ is equisliceable with respect to $g_{\eta''}$ with the measure of each slice being equal to $\mu_{prod}(\mathcal{A}_j)$ (i.e. $\mu_{\eta''^c}\Big ( \pi_{\eta''^c} \big ( \mathcal{A}_j \cap P_{g_{\eta''}}  \big ) \Big ) = \mu_{prod}(\mathcal{A}_j) $), we have for all $g_{\eta''} \in G_{\eta''}$ that

\ben
\Big \vert \, \, \, \bar{f} \big((w \circ w') (g_{\eta''}) \big)   -  \frac{1}{\mu_{prod}(\mathcal{A}_j) }  \int_{ \pi_{\eta''^c} \big ( \mathcal{A}_j \cap P_{g_{\eta''}}  \big ) } dg_{\eta''^c} \tilde{f}      \, \, \,  \Big \vert < \epsilon,
\een

and thus that

\ben
\Big \vert \, \, \,  \frac{1}{\mu_{prod} (\mathcal{A}_j)} \int_{\mathcal{A}_j} \tilde{f} d \mu_{prod} - \int_{G_{\eta''}} dg_{\eta''} \bar{f} \big((w \circ w') (g_{\eta''}) \big) \, \, \, \Big | < \epsilon
\een

Since this is true for all $j$'s eventually, and since $$\int_{G_{\eta''}} dg_{\eta''} \bar{f} \big((w \circ w') (g_{\eta''}) \big) = \int_{\mathcal{A}} f d \mu_{A-L},$$ we have that 

\ben
\Big \vert  L \Big ( \frac{1}{\mu_{prod} (\mathcal{A}_j )} \int_{\mathcal{A}_j} \tilde{f} d \mu_{prod} \Big ) - \int_{\mathcal{A}} f d \mu_{A-L} \Big \vert < \epsilon.
\een

This equation holds for all $\epsilon > 0$, which means that 

\ben
L \Big ( \frac{1}{\mu_{prod} (\mathcal{A}_j )} \int_{\mathcal{A}_j} \tilde{f} d \mu_{prod} \Big ) = \int_{\mathcal{A}} f d \mu_{A-L},
\een
and the proof is complete.
\end{proof}

It remains to prove Lemma 3:

\begin{proof}[Proof of Lemma 3]
Assume that the required neighborhood does not exist. Then for every neighborhood of the identity $U$, there is a point in $\mathcal{G}$, $x_U$, such that $x_U \cdot U$ is not a subset of any element of $\{ O_j \}_{j \in J}$. Ordering the neighborhoods of identity by inclusion we obtain a net of points $x_U$. Since $\mathcal{G}$ is compact, this net has a convergent subnet, which converges to some point $x_0 \in O_{j_0}$. Let $W$ be a neighborhood of the identity such that $x_0 \cdot W \subset O_{j_0}$, and let $W'$ be a neighborhood of the identity such that $W' \cdot W' \subset W$ (both $W$ and $W'$ exist because multiplication is continuous). Pick a neighborhood $V$, such that $x_V \in x_0 \cdot W'$, and $V \subset W'$. Then, since $W' \cdot W' \subset W$, we have $x_0 \cdot W' \cdot W' \subset x_0 \cdot W$. Therefore, we have that $x_V \cdot W' \subset x_0 \cdot W$ and thus $x_V \cdot V \subset x_0 \cdot W$. But this implies that $x_V \cdot V \subset O_{j_0}$ which is a contradiction.
\end{proof}

We have thus demonstrated that we can obtain the A-L measure from the product one. The reader should note that the procedure above, constructing a measure via restriction, is not limited to the case described and can be adapted to other situations with little effort.\\

Let us show now that $\mu_{prod}$ and $\mu_{A-L}$ have analogous support properties. To this end put a topology on $\mathcal{L}$ which makes it into a topological group such that there is a sequence of loops $\{\gamma_n\}_{n=1}^\infty$ converging to the identity such that any finite subcollection is holonomically independent (note that this is only possible if the dimension of $M$ is at least two. In the one dimensional case, it follows from the results of \cite{tlas} that $\mathcal{L}$ is either trivial if $M \simeq \mathbb{R}$, or is discrete if $M \simeq S^1$. In either case, the identity element is isolated.). An example of such a topology can be obtained by first giving $\tilde{\mathcal{L}}$ the Schwartz topology and then taking the quotient topology on $\mathcal{L}$ (an example of the needed sequence of loops is the set of circles of the Hawaiian earring). The Schwartz topology is the topology defined as follows: Embed the manifold $M$ into some $\mathbb{R}^k$ (which is always possible by Whitney's theorem), then the Schwartz topology on $\bar{\mathcal{L}}$ is the topology generated by the family of semi-norms $\{p_n\}_{n=0}^\infty$ where $p_n (\gamma) = \txt{sup}_{t \in I} || \gamma^{(n)} (t) ||$, where $|| \cdot ||$ stands for the usual Euclidean norm on the ambient $\mathbb{R}^k$. Needless to say, any weaker topology which makes $\mathcal{L}$ into a topological group will also satisfy the needed requirement.\\

We now have

\begin{lemma}
If $\mathcal{L}$ is given a topology as above, then the $\mu_{prod}$ measure of the set of functions in $\mathcal{F}$ continuous at some fixed point (e.g. the identity in $\mathcal{L}$) is zero.
\end{lemma}
 
\begin{proof}
Let the sequence of loops converging to the identity, such that each finite subcollection is holonomically independent, be $\{\gamma_n\}_{n=1}^\infty$. Let $U$ be an open neighbourhood of the identity element in $G$ whose Haar measure is strictly less than 1. Since we assume that $G$ is compact then $G$ has a finite cover by translates of $U$. Let us denote the elements of this cover by $\{U_1, \dots, U_k \}$. Note that the set of elements of $\mathcal{F}$ continuous at the identity of $\mathcal{L}$ is a subset of 
$$
\bigcup_{i=1}^k \Bigg [ \bigcup_{n=1}^\infty \Bigg (  \prod_{\gamma_n, \gamma_{n+1} \dots} U_i \times \prod_{\gamma \neq \gamma_n, \gamma_{n+1}, \dots} G  \Bigg )  \Bigg ]
$$

But each one of the sets in the brackets above has vanishing $\mu_{prod}$ measure.
\end{proof}

It is natural to expect, in view of the close relationship between the two measures, that $\mu_{A-L}$ also satisfies this property. In fact we can say a little more as the following theorem shows:

\begin{theorem}
If $\mathcal{L}$ is given a topology as described above, then the measure $\mu_{A-L}$ is supported on nowhere-continuous elements of $\mathcal{A}$. 
\end{theorem}

\begin{proof}
It is obvious that a homomorphism between topological groups is continuous if and only if it is continuous at the identity. Take a sequence of loops $\{\gamma_n\}_{n=1}^\infty$ converging to the identity such that any finite subcollection is holonomically independent. Take open $U, V \subset G$, such that $\overline{U} \subset V$, $U$ contains the identity of $G$, and the Haar measure of $V$, $|V|$ is strictly less than $1$. Note that the set of $G$-valued functions which are continuous at the identity is a subset of $\cup_{n=1}^\infty W_n$ where $W_n$ is equal to the product of $U$'s for all $\gamma_i$'s, $i \geq n$ and $G$'s for all the other loops. $W_n$ is clearly in the $\sigma$-algebra generated by the open sets. Let $f:G \to \mathbb{R}$ be a positive continuous function which is equal to 1 on $\overline{U}$, 0 on $G-V$ and is less than or equal to 1 everywhere else. If we fix $n$ for now, let $F_m : \mathcal{A} \to \mathbb{R}$ be equal to $f(g_n) \cdot {\dots} \cdot f(g_{n+m})$ where $g_i$ is the holonomy around $\gamma_i$. Then, since $F_m$ is cylindrical, we have 
\besn
\mu_{A-L} (W_n) & \leq & \int d\mu_{A-L} F_m \\
& = & \int dg_n \dots dg_{n+m} f(g_n) \dots f(g_{n+m}) \\ 
& \leq & | V |^m \to 0 \qquad \txt{as} \qquad m \to \infty.
\eesn
 
Thus $\mu_{A-L}(W_n) = 0$ and $\mu_{A-L} \big ( \cup_{n=1}^\infty W_n \big ) = 0$. 
\end{proof}

A related aspect of the support of the A-L measure was demonstrated in \cite{thiemann}. There it was shown that the measure is supported on connections which have nowhere-continuous (as a function of the endpoint) parallel transport along a given path.\\

Note that Theorem 2 can be strengthened significantly if one combines it with the known results on the automatic continuity of homomorphisms \cite{rosendal}. More precisely, the continuity of a homomorphism often follows from a weaker requirement on the map, e.g. that of being Baire. We can thus conclude that (under various conditions \cite{rosendal}) the measure $\mu_{A-L}$ is supported on non-Baire elements of $\mathcal{A}$.   \\

\textbf{Acknowledgements:} The author would like to thank an anonymous referee for a very thorough and detailed list of suggestions and remarks. Taking them into account have greatly improved the manuscript and its readability.\\

\texttt{{\footnotesize Department of Mathematics, American University of Beirut, Beirut, Lebanon.}
}\\ \texttt{\footnotesize{Email address}} : \textbf{\footnotesize{tamer.tlas@aub.edu.lb}} 

\end{document}